\documentclass[12pt]{amsart}
\usepackage{paralist}
\usepackage[%breaklinks=true, % still buggy under Linux with \stackrel
            bookmarks=true,
            bookmarksopen=true,
            pagebackref=true,
            hyperindex=true,
            colorlinks=true,
            linkcolor=blue,
            citecolor=blue,
            filecolor=blue,
            urlcolor=blue,
            ]{hyperref}
            
\hypersetup{pdfauthor   = {Lukas\ Kuehne,\ \ Rudi\ Pendavingh,\ and\ Geva\ Yashfe},
	pdftitle    = {Von Staudt Constructions for Skew-Linear and Multilinear Matroids},
	pdfsubject  = {},
	pdfkeywords = {matroids;\ multilinear matroids;\ $c$-arrangements;\ undecidability;\ Dowling geometry}}

\usepackage[utf8]{inputenc}
\usepackage[T1]{fontenc}
\usepackage[a4paper, includeheadfoot, margin=2.81cm, bottom=1cm]{geometry}                % See geometry.pdf to learn the layout options. There are lots.
\usepackage{times}
\usepackage{mathrsfs}
\usepackage{mathtools}
\usepackage{latexsym}
\usepackage{amssymb}
\usepackage{amsthm}
\usepackage{epsfig}
\usepackage{colortbl}
\usepackage{fancyvrb}
\usepackage{graphicx}
\usepackage{subcaption}
\usepackage[font=small,labelfont=bf]{caption}
\usepackage[dvipsnames]{xcolor}
\usepackage{accents} %% \undertilde
\usepackage{enumerate}
\usepackage{blkarray, bigstrut}
\usepackage{xparse}

\usepackage{wrapfig}
\usepackage{tikz}
\usetikzlibrary{shapes,arrows,matrix,backgrounds,positioning,plotmarks,calc,patterns,
decorations.shapes,
decorations.fractals,
decorations.markings,
decorations.pathreplacing,
decorations.pathmorphing,
decorations.text
}
\usepackage[colorinlistoftodos,shadow]{todonotes}
\usepackage{bm}
\usepackage{multirow}
\usepackage{mdwlist}
\usepackage{stmaryrd}
\usepackage{mathdots} %for iddots

\usepackage[toc,page]{appendix}
\usepackage{float}

\usepackage{bigfoot}
\usepackage[sort&compress,capitalise]{cleveref}

\crefname{exmp}{Example}{Examples}

\usepackage[linesnumbered,commentsnumbered,ruled,vlined]{algorithm2e}

\SetCommentSty{mycommfont}

\newtheoremstyle{mytheoremstyle} % name
    {5pt}                    % Space above
    {5pt}                    % Space below
    {\itshape}                   % Body font
    {\parindent}                           % Indent amount (empty = no indent, \parindent = para indent)
    {\bf}                   % Theorem head font
    {.}                          % Punctuation after theorem head
    {.5em}                       % Space after thm head: " " = normal interword space; \newline = linebreak
    {}  % Theorem head spec (can be left empty, meaning ‘normal’)

\theoremstyle{mytheoremstyle}

\newtheorem{theorem}{Theorem}[section]

\newtheorem{lemm}[theorem]{Lemma}
\newtheorem{prop}[theorem]{Proposition}

\makeatletter
\@addtoreset{claims}{section}
\makeatother

\newtheoremstyle{mytdefintionstyle} % name
    {5pt}                    % Space above
    {5pt}                    % Space below
    {\rm}                   % Body font
    {\parindent}                           % Indent amount
    {\bf}                   % Theorem head font
    {.}                          % Punctuation after theorem head
    {.5em}                       % Space after theorem head
    {}  % Theorem head spec (can be left empty, meaning ‘normal’)

\theoremstyle{remark}
\newtheorem{rmrk}[theorem]{Remark}

\theoremstyle{mytdefintionstyle}
\newtheorem{defn}[theorem]{Definition}

\newtheoremstyle{exmp_contd}
    {5pt}                    % Space above
    {5pt}                    % Space below
    {\rm}                   % Body font
    {\parindent}                           % Indent amount
    {\bf}                   % Theorem head font
    {.}                          % Punctuation after theorem head
    {.5em}                       % Space after theorem head
    {\thmname{#1}\ \thmnumber{ #2}\thmnote{#3}\ (continued)}  % Theorem head spec (can be left empty, meaning ‘normal’)
\theoremstyle{exmp_contd}

\DeclareMathOperator{\tr}{tr}

\newcommand\F{\mathbb{F}}

\newcommand{\Z}{\mathbb{Z}}

\renewcommand\phi{\varphi}
\newcommand{\M}{\mathcal{M}}
\newcommand{\PP}{\mathcal{P}}

\DeclareMathOperator\rk{rank}

\definecolor{darkgray}{rgb}{0.3,0.3,0.3}
\definecolor{LightGray}{gray}{0.9}

\newcommand{\topstrut}[1][1.2ex]{\setlength\bigstrutjot{#1}{\bigstrut[t]}}
\newcommand{\botstrut}[1][0.9ex]{\setlength\bigstrutjot{#1}{\bigstrut[b]}}

\DeclareDocumentEnvironment{bheadmatrix}{O{}O{*{10}{c}}}{%
	\begin{blockarray}{#2}
		#1 \\[-0.8ex]
		\begin{block}{[#2]}
			\topstrut}%
		{%
			\botstrut\\
		\end{block}
	\end{blockarray}
}%

\definecolor{darkgreen}{rgb}{0.008,0.617,0.067}
\definecolor{brown}{rgb}{0.6,0.4,0.2}

\newif\ifjournalversion

\usepackage{cite}
\makeatletter
\newcommand{\citecomment}[2][]{\citen{#2}#1\citevar}
\newcommand{\citeone}[1]{\citecomment{#1}}
\newcommand{\citetwo}[2][]{\citecomment[,~#1]{#2}}
\newcommand{\citevar}{\@ifnextchar\bgroup{;~\citeone}{\@ifnextchar[{;~\citetwo}{]}}}
\newcommand{\citefirst}{\@ifnextchar\bgroup{\citeone}{\@ifnextchar[{\citetwo}{]}}}
\newcommand{\cites}{[\citefirst}
\makeatother

\author{Lukas K\"uhne}
\address{Einstein Institute of Mathematics, The Hebrew University of Jerusalem, Giv’at Ram, Jerusalem, 91904, Israel}
\address{Max Planck Institute for Mathematics in the Sciences, Inselstr. 22, 04103, Leipzig, Germany}
\email{\href{mailto:Lukas Kuehne<lukas.kuhne@mis.mpg.de>}{lukas.kuhne@mis.mpg.de}}

\author{Rudi Pendavingh}
\address{Department of Mathematics and Computer Science, Eindhoven University of Technology, 5600 MB Eindhoven, Netherlands}
\email{\href{mailto:Rudi Pendavingh<r.a.pendavingh@tue.nl>}{r.a.pendavingh@tue.nl}}

\author{Geva Yashfe}
\address{Einstein Institute of Mathematics, The Hebrew University of Jerusalem, Giv’at Ram, 9190401 Jerusalem, Israel}
\email{\href{mailto:Geva Yashfe <geva.yashfe@mail.huji.ac.il >}{geva.yashfe@mail.huji.ac.il}}

\begin{document}

\title{Von Staudt Constructions for Skew-Linear and Multilinear Matroids}
\begin{abstract}
This paper compares skew-linear and multilinear matroid representations.
These are matroids that are representable over division rings and (roughly speaking) invertible matrices, respectively.
The main tool is the von Staudt construction, by which we translate our problems to algebra.
After giving an exposition of a simple variant of the von Staudt construction we present the following results:

\begin{itemize}
	\item Undecidability of several matroid representation problems over division rings.

	\item An example of a matroid with an infinite multilinear characteristic set, but which is not multilinear in characteristic $0$.

	\item An example of a skew-linear matroid that is not multilinear.
\end{itemize}
\end{abstract}

\thanks{L.K. was supported by a Minerva fellowship of the Max-Planck-Society, the Studienstiftung des deutschen Volkes and by ERC StG 716424 - CASe.
G.Y. was supported by ERC StG 716424 - CASe and by ISF grant 1050/16.}

\maketitle

\section{Introduction}
The thread of this paper winds around the von Staudt construction. We collect some examples and theorems which discuss and compare the basic properties of skew-linear and multilinear matroid representations. With the von Staudt construction, this reduces to a sequence of questions about matrix rings and division rings.

Skew-linear matroids are those representable over a division ring; multilinear matroids are, roughly speaking, those representable in invertible matrices. Definitions appear in \cref{sec:prelim}, while \cref{sec:related_work} describes some appearances of multilinear matroids in mathematics and computer science.

\subsection{Results}
Our main results are:
\begin{enumerate}
	\item There is no algorithm to recognize skew-linear matroids\footnote{The result is somewhat finer: we show that there is a division ring $D$ such that it is undecidable whether a matroid is $D$-linear, and that it is also undecidable whether a matroid is representable over a division ring of a given characteristic.}.
	\item It is possible for a matroid which is not multilinear over a field of characteristic $0$ to have an infinite (multilinear) characteristic set. However, if the multilinear characteristic set of a matroid contains $0$ then it is infinite.
	\item Not every skew-linear matroid is multilinear.
\end{enumerate}

These theorems can be compared to the following known facts, whose proofs are rather different despite the similar statements:
\begin{enumerate}
	\item There is no algorithm to recognize multilinear matroids \cite{KY19}.
	\item A matroid with infinite skew-linear characteristic set can be represented over a skew-field of characteristic $0$. However, there is a matroid with skew-linear characteristic set $\{0\}$ \cite{EH91}.
	\item Not every multilinear matroid is skew-linear \cite{PvZ13}.
\end{enumerate}

The main tool used throughout the paper is a construction essentially due to von Staudt (\cite{vS57}), which reduces the solvability of a system of polynomial equations to a sequence of matroid representation problems. This works in both the skew-linear and the multilinear settings (the skew-linear case is classical). We work with a simple version of the construction, and provide detailed exposition.

Our results are proved by use of von Staudt constructions in conjunction with the following algebraic theorems:
\begin{enumerate}
	\item The word problem for division rings is undecidable (and slight refinements of this theorem) \cite{Mac73}.
	\item The Weyl algebra over a field of characteristic $0$ has no nontrivial finite-dimensional representations. However, the Weyl algebra over a field of characteristic $p$ is finite-dimensional over its center (see \cite{Et11} for example).
	\item There is a system of polynomial equations in noncommuting variables that has a solution in some division ring, but which cannot be solved in matrices over a field (\cref{thm:bs_polynomial_system} -- this seems to be new).
\end{enumerate}

\subsection{Related work}\label{sec:related_work}
This paper is related to \cite{PvZ13}, in which the second author and van Zwam define skew partial fields and matroid representations over them. This is a simultaneous generalization of both classes of matroid representations considered here. Several characterizations of representability are given in that paper, and it is proved that not every multilinear matroid is skew-linear. We answer a question posed there by showing that not every skew-linear matroid is multilinear, either.

Parts of this paper are also parallel to a part of \cite{EH91}, in which the authors describe a variant of the von Staudt construction over division rings and give some examples of skew-linear matroids with interesting characteristic sets.

Our results on undecidability owe their existence to Macintyre's paper \cite{Mac73} on the word problem in division rings. The results in that paper were improved and refined in \cite{Mac79}, but we do not use the latter paper here.

It is possible that multilinear matroids first appeared in \cite{GM88}, Goresky and Macpherson's work on stratified Morse theory. They defined $c$-arrangements (objects dual to multilinear matroids) as examples of subspace arrangements to be studied from a topological viewpoint.

Multilinear matroids also appear in cryptography and network coding: in cryptography, their ports are access structures of perfect ideal secret sharing schemes (it is not known if this construction yields all such access structures). See \cite{SA98,BBP14} for details. 
In network coding, the multilinear representability problem is equivalent to certain network capacity problems, in which only linear coding functions are permitted: see \cite{ElR10}, compare also~\cite{DFZ07}.

\subsection{History and applications of the von Staudt construction}\label{sec:history}
The authors are not historians; this subsection reflects their point of view and their interaction with this circle of ideas.

In \cite{vS57}, von Staudt introduced the algebra of throws. This is a geometric construction, based on the cross-ratio, for adding and multiplying points on a projective line. Using it, polynomial algebraic relations can be translated into corresponding point-and-line configurations (cf. \cite{RG11,VY65}).

This construction and its variants can be used to prove the coordinatization theorem of projective geometry, namely that a Desarguesian projective plane is isomorphic to the projective plane over a division ring (see \cite{VY65} for example). Together with Pascal's theorem, it is also the tool used to prove that Hilbert's axioms for plane Euclidean geometry accurately capture the notion of a real inner product space of dimension $2$ \cite{Hil02}.

In matroid theory, it has perhaps most notably been applied by Mn\"{e}v to prove his universality theorem for the realization spaces of oriented matroids~\cite{Mne88}.
One version of this theorem states that the space of solutions to a system of real polynomial equalities and strong inequalities\footnote{The solution space to such a system is called a basic primary semialgebraic set.} is homotopy equivalent to the realization space of an oriented matroid. This is a highly technical result, and the use of the construction requires great care.

The von Staudt construction is also closely related to Dowling geometries, which are essentially special cases in which only multiplication is used. See \cite{BBP14,KY19} for two related applications.

\subsection{Outline of the paper}
In \cref{sec:prelim} we recall the definitions of multilinear and skew-linear matroid representations, provide a short discussion of the projective plane over a division ring (which is later used in the von Staudt construction,) and define projective equivalence of representations.

In \cref{sec:VS} we describe a simple variant of the von Staudt construction, and prove the basic theorem relating polynomial equations in noncommuting variables with representation problems for the matroids produced by the construction.

\cref{sec:undec} is devoted to the undecidability of certain problems regarding skew-linearity of matroids.

\cref{sec:weyl} discusses characteristic sets and constructs an example of a matroid which is multilinear over all prime characteristics, but not over characteristic $0$.

Finally, in \cref{sec:baumslag} we prove that not every skew-linear matroid is multilinear.
also relevant to network coding problems:

\section*{Acknowledgments}
The first author would like to thank Oren Becker for helpful discussions of the Weyl algebra.

\section{Multilinear and skew-linear representations of matroids}\label{sec:prelim}

We use the following notation: if $V$ is a vector space over a field $\F$ we denote by $\mathrm{Gr}(c,V)$ the family of $c$-dimensional vector subspaces of $V$. 
We call subspaces $W_1,\ldots,W_n$ of $V$ \emph{independent} if \[\dim\sum_{i=1}^n W_i = \sum_{i=1}^n \dim W_i.\]
If $D$ is a division ring, all modules over $D$ have $D$ acting from the right. In particular, $D^r$ always denotes the right vector space of dimension $r$ over $D$.

\begin{defn}\label{def:rep}
	Let $M$ be a matroid of rank $r$ on the ground set $E$.
	\begin{enumerate}
		\item[(i)] A \emph{representation of $M$ over a division ring $D$} is a function $E\rightarrow A$, where $A$ is a right module over $D$, such that a subset of $E$ is independent in $M$ if and only if its image is skew-linearly independent in $A$.
		\item[(ii)] A \emph{multilinear representation} of $M$ over a field $\F$ is given by a $c\in\mathbb{N}$, a vector space $V$ over $\F$, and a function from $E$ to $\{0\}\cup\mathrm{Gr}(c,V)$ (where $0$ denotes the trivial subspace of $V$) such that:
		\begin{itemize}
			\item A subset of $E$ is independent if and only if its image is independent.
			\item If $W_1,\ldots,W_t$ is a subset of the image of $E$ in $\mathrm{Gr}(c,V)$, the sum $\sum_i W_i$ has dimension a multiple of $c$.
		\end{itemize}
	\end{enumerate}
\end{defn}

If a matroid has a representation over some division ring, we call it \emph{skew-linear} or division ring representable. If it has a multilinear representation (of order $c$) we say it is representable as a \emph{$c$-arrangement}, or call it multilinear.

We can write these definitions in coordinates: the module $A$ above can be replaced by~$D^r$ (written as a column space, with $D$ acting from the right). The vector space $V$ can be replaced with $\F^{cr}$, and $\mathrm{Gr}(c,V)$ can be replaced with the set of matrices in $M_{rc,c}(\F)$ having rank $c$ (the subspace corresponding to such a matrix is its column span).

Concretely, in coordinates a representation can be given by a matrix. Suppose $M$ is a matroid as above and $E=[n]$. Then a representation of $M$ over a division ring $D$ can be given by a matrix in $M_{r,n}(D)$, where the $j$-th column is the image of $j\in E$. Similarly, a multilinear representation of order $c$ over $\F$ can be given by a matrix in $M_{rc,cn}(\F)$: we think of this as a block matrix with blocks of size $c\times c$, so the $j$-th block column has size $r\times 1$ (in blocks) and its column span is the image of $j\in E$. 

We can then discuss the column rank of matrices instead of skew-linear or linear independence. For example, in this language, a multilinear representation of order $c$ is a matrix in $M_{rc,cn}(\F)$ such that for a subset of $E$ having rank $s$, the minor consisting of the corresponding block columns of the matrix has rank $c\cdot s$.

\subsection{Projective equivalence of representations}
Suppose $M$ is represented over a division ring, or as a $c$-arrangement. A change of coordinates in the ambient module or vector space gives rise to a different representation which is nevertheless the same in every essential sense. Projective equivalence is the equivalence relation on matroid representations which captures this idea (for the case of representations over a field, see \cite[Section 6.3]{Oxl11}). We will use such equivalences to bring representations to a more convenient form.

\begin{defn}
	Let $D$ be a division ring and let $B\in M_{r,n}(D)$. A matrix $B'\in M_{r,n}(D)$ is projectively equivalent to $B$ if it can be obtained from $B$ by steps of the form:
	\begin{enumerate}[(i)]
		\item Multiplying a column of $B$ by an invertible scalar (from the right).
		\item Multiplying $B$ by an invertible $r\times r$ matrix from the left.
	\end{enumerate}
\end{defn}

It is clear that if $B$ is a representation matrix of a matroid $M$, any $B'$ projectively equivalent to $B$ also represents $M$.

For matroid representations in a general module $A$ over $D$, these steps correspond to applying $D$-linear automorphisms of $A$ and scaling individual elements in the image by an invertible scalar.

The following is the multilinear analogue:

\begin{defn}
	Let $\F$ be a field and let $B\in M_{rc,cn}(\F)$ be a multilinear representation matrix of $M$ over $\F$. A matrix $B'\in M_{rc,cn}(D)$ is a representation of $M$ projectively equivalent to $B$ if it is obtained from $B$ by steps of the form:
	\begin{enumerate}[(i)]
		\item Multiplying a block column of $B$ by an invertible matrix (from the right).
		\item Multiplying $B$ by an invertible $rc\times rc$ matrix from the left.
	\end{enumerate}
\end{defn}

Note that multiplying a block column by an invertible matrix from the right does not change its column span. Equivalently, the coordinate-free description of projective equivalence for multilinear representations involves only automorphisms of the ambient vector space.

\subsection{A brief review of the projective plane}

We review some definitions and set up notation for projective planes coordinatized by division rings. A reference for this section is~\cite[Chapter 6]{Har67}.

The (right) projective plane $D\mathbb{P}^{2}$ over a division ring
$D$ has points $(D^3 \setminus \{(0,0,0)\})/\sim$, 
where $\sim$ is the equivalence relation in which
$(a,b,c)\sim (a',b',c')$ if and only if there exists $\lambda\in D\setminus\{0\}$ such that $(a\lambda,b\lambda,c\lambda)=(a',b',c')$.
The equivalence class of $(a,b,c)$ is denoted $\left[a:b:c\right]$.

A projective line is a subset of the form
\[
\left\{ \left[a\lambda+a^{\prime}\mu:b\lambda+b^{\prime}\mu:c\lambda+c^{\prime}\mu\right]\mid\lambda,\mu\in D,\text{ not both zero}\right\} ,
\]
where $\left(a,b,c\right)$ and $\left(a^{\prime},b^{\prime},c^{\prime}\right)\in D^{3}$
are linearly independent vectors. 

Thus the points of $D\mathbb{P}^{2}$
are equivalence classes, and each equivalence class is a one-dimensional
subspace of $D^{3}$, minus the zero vector.
In this language a projective line is the set of one-dimensional subspaces
contained in a two-dimensional subspace of $D^{3}$.

A projective line is determined by any two of its points, and any
two lines intersect at a unique point. A projective transformation
of $D\mathbb{P}^{2}$ is a bijection $D\mathbb{P}^{2}\rightarrow D\mathbb{P}^{2}$
such that the image and preimage of every projective line is again
a projective line. One part of the fundamental theorem of projective geometry
states that any such transformation is given by composing an automorphism
of $D$ with a map defined on representatives in $D^{3}$ by 
\[
[x,y,z]\mapsto\left(A\left[\begin{matrix}x\\
y\\
z
\end{matrix}\right]\right)^{T}
\]
where $A\in M_{3}(D)$ is an invertible matrix.

The affine plane can be embedded in the projective plane by mapping
a point $(x,y)$ to $\left[1:x:y\right]$. This map takes affine lines
into projective lines. Two affine lines are parallel if and only if
their image intersects on the line at infinity, which consists precisely
of those points of $D\mathbb{P}^{2}$ which are not in the image of
the embedding.

\subsubsection{Notation}

Our convention is to denote by $O$ (for origin) the point $\left[1:0:0\right]$,
and to call the line spanned by $x_{\infty}\coloneqq\left[0:1:0\right]$
and $y_{\infty}\coloneqq\left[0:0:1\right]$ the line at infinity.
We further denote $x_{1}\coloneqq\left[1:1:0\right]$ and $y_{1}\coloneqq\left[1:0:1\right]$.

\subsubsection{\texorpdfstring{Matroid representations and $D\mathbb{P}^2$}{Matroid representations and the projective plane}}

Let $M$ be a loop-free matroid of rank $3$ on the ground set $E=[n]$. A
representation of $M$ over a division ring $D$ is essentially a
configuration of points $\{p_{i}\}_{i\in[n]}$ in $D\mathbb{P}^{2}$.
Each column $\left[\begin{smallmatrix}x\\
y\\
z
\end{smallmatrix}\right]$ of the matrix of the representation corresponds to a point $\left[x:y:z\right]$
of $D\mathbb{P}^{2}$, and three points of $M$ are dependent precisely
when the corresponding points in $D\mathbb{P}^{2}$ lie on a line.
Under this correspondence, a set of parallel elements of $M$ is a
single point of $D\mathbb{P}^{2}$, indexed several times. 

Note that passing from nonzero elements of $D^{3}$ to their images in the quotient $D\mathbb{P}^{2}$
causes essentially no loss of information. This is because multiplying
a column of the matrix of a representation by a scalar from the right
gives a projectively equivalent representation.

We will use the following lemma.
\begin{lemm}
	\label[lemma]{lem:proj_equiv} If $\ell_{1}$ and $\ell_{2}$ are
	two lines of $D\mathbb{P}^{2}$ intersecting at a point $\overline{O}$, and
	$\overline{a},\overline{b}\in\ell_{1}$, $\overline{c},\overline{d}\in\ell_{2}$ are points distinct from each
	other and $\overline{O}$, there exists a projective transformation taking $\overline{O},\overline{a},\overline{b},\overline{c},\overline{d}$
	to
	\[
	\left[1:0:0\right],\left[1:1:0\right],\left[0:1:0\right],\left[1:0:1\right],\left[0:0:1\right],
	\]
	respectively.
\end{lemm}
\begin{rmrk}\label[remark]{rem:2.4_multilinear}
This can easily be proved using the fundamental theorem of projective geometry, together with the fact that $\{\bar{a},\ldots,\bar{d}\}$ is a projective frame. We give a computational proof, which also applies to the multilinear analogue in which each point $[x:y:z]\in D^3$ is replaced by the column span of the block matrix $\begin{bsmallmatrix}X \\ Y \\ Z\end{bsmallmatrix}$, where $X,Y,$ and $Z$ are $c\times c$ matrices over a field.
The analogy between the projective and the multilinear case can be formalized: one can say that a projective space over a matrix ring is a Grassmannian, and define the Grassmannian, by analogy with the projective space, as a quotient of the collection of matrices of a given size which have full column rank.
\end{rmrk}
\begin{proof}
	Take representatives $O,\ldots,d$ of the points
	in $D^{3}$. We have $a=O\lambda+b\mu$
	and $c=O\lambda^{\prime}+d\mu^{\prime}$ for some $\lambda,\mu,\lambda^{\prime},\mu^{\prime} \in D$.
	
	By replacing $b$ with the representative $b\mu\lambda^{-1}$
	of $\overline{b}$, we may assume $\mu=\lambda$ without loss of generality,
	and similarly $\mu^{\prime}=\lambda^{\prime}$. 
	
	Writing the points $O,\ldots,d$ as column vectors,
	multiplying from the left by the left inverse of the matrix
	\[
	\left(\begin{matrix}| & | & |\\
	O & b & d\\
	| & | & |
	\end{matrix}\right)
	\]
	(with columns $O,b,d$, respectively)
	takes $O,b,d$ to $\begin{bsmallmatrix}1\\
	0\\
	0
	\end{bsmallmatrix},\begin{bsmallmatrix}0\\
	1\\
	0
	\end{bsmallmatrix},\begin{bsmallmatrix}0\\
	0\\
	1
	\end{bsmallmatrix},$and $a,c$ to $\begin{bsmallmatrix}\lambda\\
	\lambda\\
	0
	\end{bsmallmatrix},\begin{bsmallmatrix}\lambda^{\prime}\\
	0\\
	\lambda^{\prime}
	\end{bsmallmatrix}$, respectively. Passing back to elements of $D\mathbb{P}^{2}$ gives
	the result.
\end{proof}

\section{Von Staudt constructions in matrices and division rings}\label{sec:VS}
In this section we describe a variant of the von Staudt construction, which gives a method for turning the solvability of a system of polynomial equations into a matroid representation problem.

\subsection{The geometric idea}

The basic geometric idea is most easily pictured in the real affine
plane: see~\Cref{fig:staudt_explanation}. We begin with two lines,
an ``$x$-axis'' and a ``$y$-axis'' intersecting at an origin
point $O$.
In~\Cref{fig:staudt_explanation}, we represent these axes by two perpendicular lines.
We choose distinguished unit points $x_{1},y_{1}$
on the axes with $\left|Ox_{1}\right|=\left|Oy_{1}\right|=1$.

First consider the situation in~\Cref{fig:parallel_multiplication}.
The segments $\overline{x_{k}y_{k}}$ and $\overline{x_{1}y_{1}}$
are parallel, implying $|Ox_{k}|=|Oy_{k}|$. Similarly, we see
\[
\frac{|Oy_{k}|}{|Oy_{1}|}=\frac{|Ox_{i}|}{|Ox_{j}|}
\]
because $\overline{x_{j}y_{1}}$ is parallel to $\overline{x_{i}y_{k}}$.
Rearranging, we find 
\[
\left|Oy_{k}\right|\cdot\left|Ox_{j}\right|=\left|Ox_{i}\right|,
\]
a multiplicative relation.

Now consider~\Cref{fig:parallel_addition}. The triangles $Oy_{1}x_{j}$
and $x_{k}r_kx_{i}$ are similar, since corresponding segments are parallel.
In fact, the two triangles are congruent because $\left|Oy_{1}\right|=\left|x_{k}r_k\right|$:
this follows from the fact that $Oy_{1}r_kx_{k}$ is a parallelogram. The
congruence of the two triangles gives $\left|x_{k}x_{i}\right|=\left|Ox_{j}\right|$,
so 
\[
\left|Ox_{i}\right|=\left|Ox_{k}\right|+\left|x_{k}x_{i}\right|=\left|Ox_{k}\right|+\left|Ox_{j}\right|.
\]
The equation $\left|Ox_{i}\right|=\left|Ox_{k}\right|+\left|Ox_{j}\right|$
is the desired additive relation.

In both cases, the given data was a point and line configuration in
which certain lines are parallel, together with the position of $x_{1}$
and $y_{1}$ on their respective lines.
Since matroid representations in rank $3$ naturally live in a projective plane, we replace the affine plane by the projective plane in the following construction.

\begin{figure}
	\begin{subfigure}[b]{0.5\linewidth}
		\centering
		\includegraphics[width=1\linewidth]{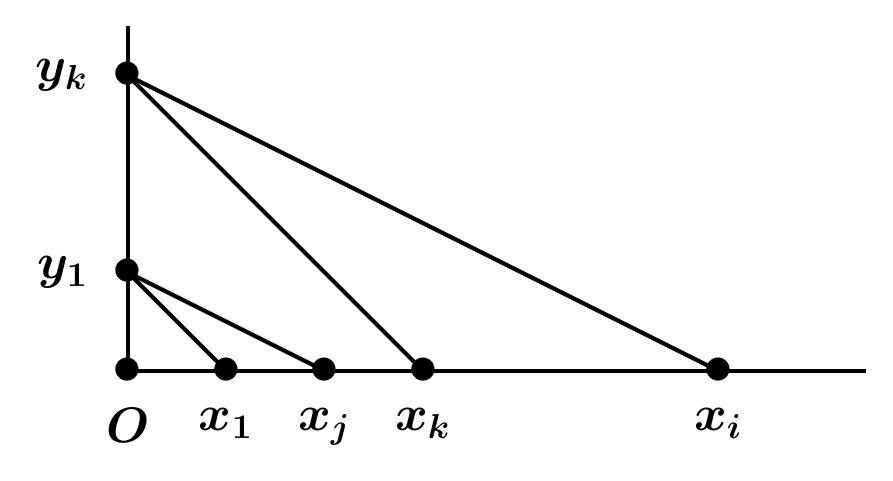} 
		\caption{A diagram exhibiting a multiplicative relation.} 
		\label{fig:parallel_multiplication} 
		%		\vspace{4ex}
	\end{subfigure}%% 
	\begin{subfigure}[b]{0.5\linewidth}
		\centering
		\includegraphics[width=1\linewidth]{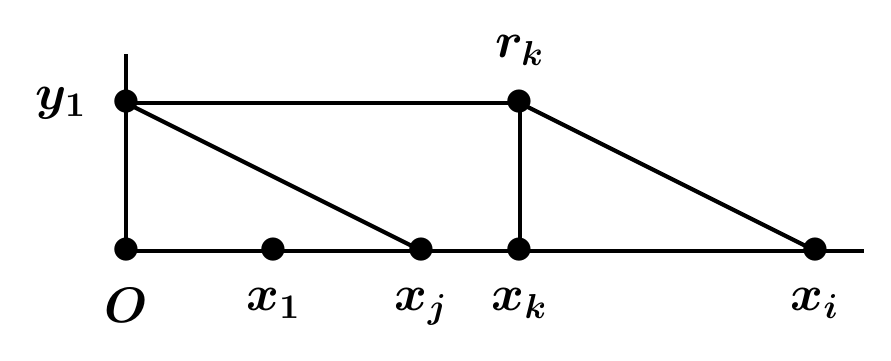} 
		\caption{A diagram exhibiting an additive relation.} 
		\label{fig:parallel_addition} 
%			\vspace{1ex}
	\end{subfigure} 
	\caption{Two diagrams motivating the following von Staudt constructions.}
	\label{fig:staudt_explanation}
\end{figure}

\subsection{The construction}\label{sec:VC_construction}

\begin{figure}
	\begin{subfigure}[b]{0.5\linewidth}
		\centering
		\includegraphics[width=1\linewidth]{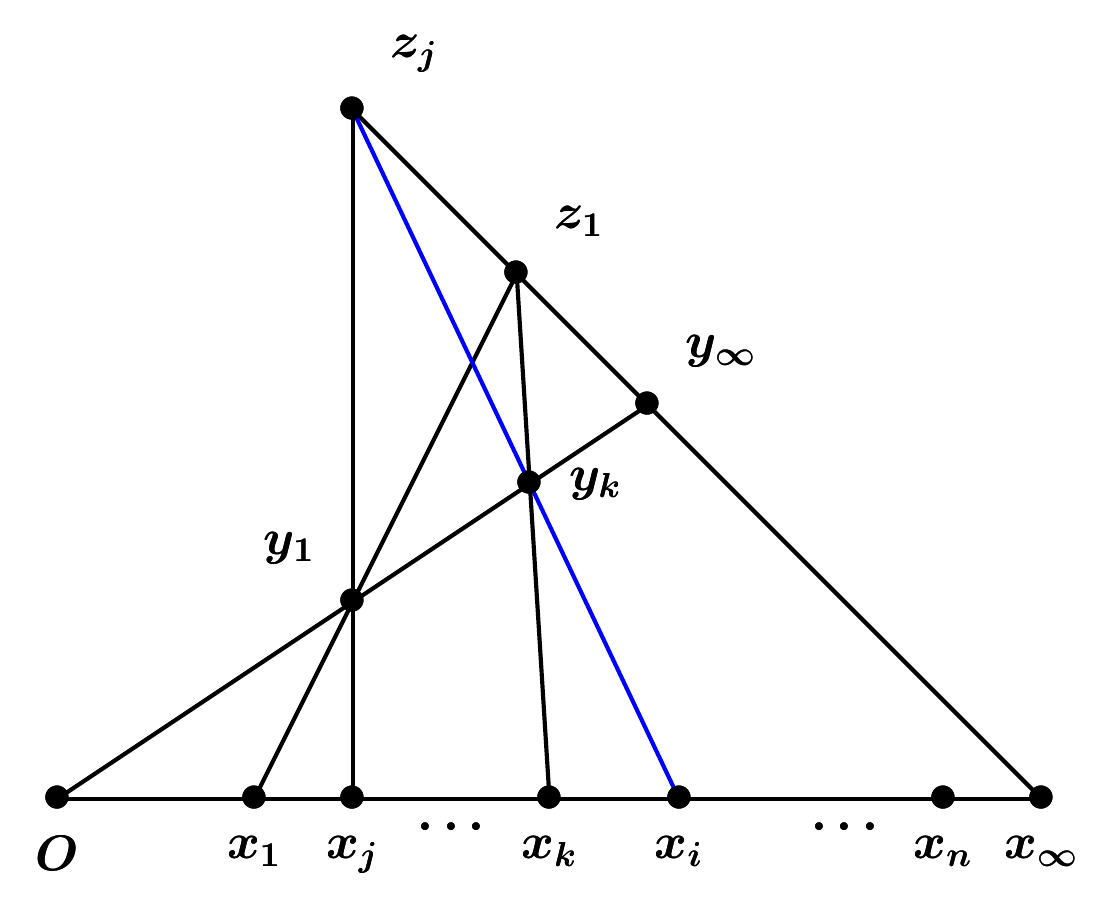} 
		\caption{Part of $M_{\PP}$ for $P$: $X_i=X_j\cdot X_k$.} 
		\label{fig:staudt_multiplication} 
		%		\vspace{4ex}
	\end{subfigure}%% 
	\begin{subfigure}[b]{0.5\linewidth}
		\centering
		\includegraphics[width=1\linewidth]{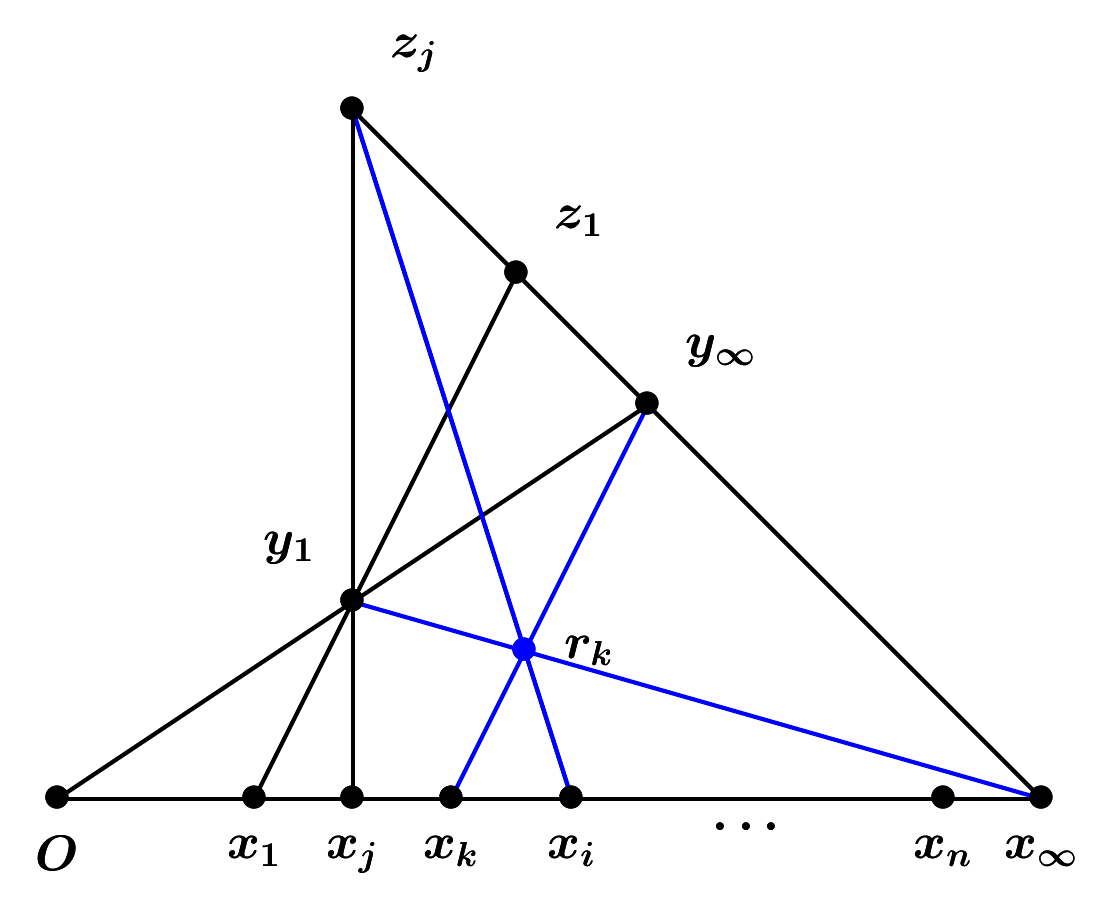} 
		\caption{Part of $M_{\PP}$ for $P$: $X_i=X_j+ X_k$.} 
		\label{fig:staudt_addition} 
		%		\vspace{4ex}
	\end{subfigure} 
	\caption{The two building blocks of the von Staudt matroid $M_{\PP}$.
		The elements and circuits corresponding to the polynomial $P$ are depicted in blue.
		These pictures correspond to the ones shown in~\Cref{fig:staudt_explanation} after adding the line at infinity.}
	\label{fig:staudt_matroid} 
\end{figure}
The von Staudt construction encodes equations of the form $X_i=X_j+X_k$ or $X_i=X_j \cdot X_k$ in the circuits of a matroid.
We begin by showing that any system of polynomial equations is equivalent to one in which each equation is of one of these forms.

Recall that $\mathbb{Z}\langle X_{1},\dots,X_{n}\rangle$ is the ring generated over $\mathbb{Z}$ by $n$ noncommuting variables. A \emph{polynomial equation} (with integer coefficients) in the variables $X_1,\ldots,X_n$ is an expression of the form $P=Q$, where $P,Q\in \mathbb{Z}\langle X_{1},\dots,X_{n}\rangle$.

Throughout this article, we assume all systems of equations to be finite.

\begin{defn}
	An equation in the variables $X_1,\ldots,X_n$ is called \emph{atomic}\footnote{In~\cite{RG99}, a related concept is called Shor normal form.} if it is of one of the following forms: 
	\begin{enumerate}
		\item $X_i=X_j+X_k$ with $1\le i,j,k\le n$ or,
		\item $X_i=X_j \cdot X_k$ with $1\le i,j,k\le n$.
	\end{enumerate}
	A system of equations is called \emph{atomic} if it consists of atomic equations together with $X_0=0$ and $X_1=1$.
\end{defn}

Let $P=Q$ be a polynomial equation in the variables $X_1,\ldots,X_n$. A \emph{solution} to this equation in a division ring $D$ is a tuple $(d_1,\ldots,d_n) \in D^n$ such that \[P(d_1,\ldots,d_n)=Q(d_1,\ldots,d_n)\] in $D$. Similarly, a solution in $c\times c$ matrices over a field $\mathbb{F}$ is a tuple $(A_1,\ldots,A_n)\in \left(M_c(\F)\right)^n$ such that $P(A_1,\ldots,A_n)=Q(A_1,\ldots,A_n)$.

\begin{lemm}\label{lem:atomic}
	Let $\mathcal{P}$ be a system of polynomial equations.
	Then there exists an $N\in\mathbb{N}$ and a system of polynomial equations $\mathcal{P}'$ in the variables $X_{0},\dots,X_{N}$ such that:
	
	\begin{enumerate}[(i)]
		\item\label{lem:atomic:i} The system $\mathcal{P}'$ is atomic.
		\item For any division ring $D$, the system $\mathcal{P}$ has a solution in $D$ if and only if $\mathcal{P}'$ has a solution in $D$.
		\item For any field $\F$ and natural number $c$, $\mathcal{P}$ has a solution in $M_c(\F)$ if and only if $\mathcal{P}'$ has a solution in $M_c(\F)$.
	\end{enumerate}
\end{lemm}
\begin{proof}
We describe an explicit way to construct the system $\mathcal{P}'$.
The statements on equivalent solvability follow immediately.

First replace any positive integer $n$ by the expression $\sum_{i=1}^n X_1$.
Replace any $0$ by the variable $X_0$.
Subsequently, replace any negative integer $m$ by a new variable $X'$ together with the equation $X_0=X'+\sum_{i=1}^{-m}X_1$.

Now consider a monomial $X_{i_{1}}\cdot\ldots\cdot X_{i_{k}}$.
Introducing new variables $X_1'\dots,X_k'$ the monomial can be replaced
by the single variable $X_{k}^{\prime}$ after adding the new equations
\[
X_{2}^{\prime}=X_{i_{1}}\cdot X_{i_{2}},\quad
X_{3}^{\prime}=X_{2}^{\prime}\cdot X_{i_{3}}, \quad
\cdots \quad
X_{k}^{\prime}=X_{k-1}^{\prime}\cdot X_{i_{k}}.
\]
Similarly, a sum of variables $X_{i_{1}}+\ldots+X_{i_{k}}$ can be
replaced by a single variable, using the same process with $\cdot$
replaced by $+$.

Lastly, add the equations $X_0=0$ and $X_1=1$ to the system $\mathcal{P}'$.
\end{proof}

Let $\mathcal{P}$ be an atomic system of equations.
We construct a collection of circuits $M_{\mathcal{P}}$ on a ground set $E_{\mathcal{P}}$ as follows. 
\begin{enumerate}
	\item Begin with the ground set
	\[E_\mathcal{P} = \{O,x_\infty,y_\infty\} \cup \{x_i\}_{i=1}^N \cup \{y_i\}_{i=1}^N \cup \{z_i\}_{i=1}^N. \]
	Define each subset of size $3$ of one of the sets:
	\[\{O,x_\infty\}\cup \{x_i\}_{i=1}^N,\qquad \{O,y_\infty\}\cup \{y_i\}_{i=1}^N,\qquad\mathrm{and}\quad \{x_\infty,y_\infty\}\cup \{z_i\}_{i=1}^N\]
	to be a circuit in $M_\mathcal{P}$.
	\item We now add circuits for each equation:
	\begin{enumerate}[i.]
		\item For each equation of the form $X_{i}=X_{j}\cdot X_{k}$, add the
		circuit $\left\{ x_{i},y_{k},z_{j}\right\} $ depicted in \Cref{fig:staudt_multiplication}. 
		
		\item For each equation of the form $X_{i}=X_{j}+X_{k}$, add an element
		$r_{k}$ to the ground set together with the circuits
		\[
		\left\{ y_{1},r_{k},x_{\infty}\right\} ,\quad\left\{ x_{k},r_{k},y_{\infty}\right\} ,\quad\left\{ x_{i},r_{k},z_{j}\right\} 
		\]
		depicted in \Cref{fig:staudt_addition}.
	\end{enumerate}
	\item Finally, define each subset of size $4$ of $E_\mathcal{P}$ which does not yet contain a circuit to be a circuit.
\end{enumerate}

\textbf{Warning.}
It is possible that $M_\mathcal{P}$ is not the family of circuits of a matroid. If $\mathcal{P}$ contains an equation of the form $X_{i}=X_{j}\cdot X_{i}$, then $M_\mathcal{P}$ contains the two circuits $\left\{ x_{i},y_{j},z_{i}\right\} ,\left\{ x_{i},y_{1},z_{i}\right\} $, but has no circuit contained in $\left\{ x_{i},y_{1},y_{j}\right\} $. Thus $M_\mathcal{P}$ does not satisfy the circuit elimination axiom.

Nevertheless, we can discuss its matroidal weak images, which are defined as follows.
\begin{defn}
	A matroid $M^{\prime}$ is a \emph{weak image} of a collection of circuits
	$M$ on the same ground set if every circuit of $M$ contains a circuit
	of $M^{\prime}$.
\end{defn}

This enables us to define the set of von Staudt matroids associated to a system of equations.
\begin{defn}\label{def:vS_matroid}
	Let $\mathcal{P}$ be an atomic system of equations.
	If $M_{\mathcal{P}}$ is a matroid, we call it \emph{the principal
		von Staudt matroid of $\mathcal{P}$}.	
	The \emph{family of von Staudt matroids associated to $\mathcal{P}$}
	is the set of matroidal weak images $M_{\mathcal{P}}^{\prime}$
	of $M_{\mathcal{P}}$ which satisfy the following conditions: 
\begin{enumerate}
	\item $M_{\mathcal{P}}^{\prime}$ is loop-free.
	\item The restrictions of $M_{\mathcal{P}}$ and $M_{\mathcal{P}}^{\prime}$
	to $\left\{ O,x_{1},y_{1},x_{\infty},y_{\infty}\right\} $ are identical:
	the five points are distinct, forming two lines which intersect at
	$O$.
	\item No element $x_{i}\neq x_{\infty}$ is parallel to $x_{\infty}$ in
	$M_{\mathcal{P}}^{\prime}$. 
\end{enumerate}
We denote this set by $\mathcal{M}_{\mathcal{P}}$.

The family $\mathcal{M}_{\mathcal{P}}$ is always finite, and can
be computed from $\mathcal{P}$. It is possible to replace $\mathcal{M}_{\mathcal{P}}$
with a family of simple matroids having the same properties by replacing
each $M\in\mathcal{M}_{\mathcal{P}}$ with its simplification, or
in other words, by deleting all but one from each maximal subset of
mutually parallel elements of each $M\in\mathcal{M}_{\mathcal{P}}$.
\end{defn}

The following theorem describes the relation between solvability of $\mathcal{P}$ and representability of members of $\mathcal{M_P}$.

\begin{theorem}\label{thm:VS}
	Let $\mathcal{P}$ be an atomic system of equations in the variables $X_0,\dots, X_N$, and let $\mathcal{M}_{\mathcal{P}}$
	be the family of von Staudt matroids associated to $\mathcal{P}$.

	\begin{enumerate}[(i)]
		\item\label{thm:i} If $\mathcal{P}$ has a solution over a division ring $D$, then at least one member of $\mathcal{M}_{\mathcal{P}}$ is representable over $D$. 
		\item\label{thm:ii} If a matroid in $\mathcal{M}_{\mathcal{P}}$ is representable over a division ring $D$ then $\mathcal{P}$ has a solution in~$D$.
		\item\label{thm:iii} If a matroid in $\mathcal{M}_{\mathcal{P}}$ has a multilinear representation of order $c$ over a field $\F$, then $\mathcal{P}$ has a solution in $M_c(\F)$, with all matrices invertible or $0$.
	\end{enumerate}
\end{theorem}

The assumption that $\mathcal{P}$ is atomic can be dropped: given a general system of equations, we can apply~\Cref{lem:atomic} to reduce to the atomic case.

\begin{proof}[Proof of~\Cref{thm:VS}~\eqref{thm:i}]
Let $D$ be any division ring.
Suppose the substitution $X_{0}=0,X_{1}=a_1=1,X_{2}=a_{2},\ldots,X_{N}=a_{N}$
is a solution of $\mathcal{P}$ for some $a_1,\dots,a_N\in D$.
Consider the following map $\rho$ from the
ground set of $M_{P}$ to $D\mathbb{P}^{2}$:
\begin{align*}
O\mapsto&\left[1:0:0\right],\\
x_{i}\mapsto&\left[1:a_{i}:0\right] \text{ for all } 1\le i \le N,\\
y_{i}\mapsto&\left[1:0:a_{i}\right]\text{ for all } 1\le i \le N,\\
x_{\infty}\mapsto&\left[0:1:0\right],\\
y_{\infty}\mapsto&\left[0:0:1\right],\\
z_{i}\mapsto&\left[0:a_i:-1\right]\text{ for all } 1\le i \le N,\\
r_{i}\mapsto&\left[1:a_{i}:1\right]\text{ for all } 1\le i \le N.
\end{align*}

We show that this map represents some matroid in $\M_{\mathcal{P}}$. 
First we need to verify that the conditions described in \Cref{def:vS_matroid} hold.

The matroid represented by $\rho$ is loop-free (there is no ``$0$ point'' in $D\mathbb{P}^2$). Similarly, no $\rho(x_i)$ is parallel to $\rho(x_\infty)$ (for $i\neq\infty$), because $\rho(x_\infty)$ is on the line at infinity and other $\rho(x_i)$ are not. The restriction of $\rho$ to $\{O,x_1,y_1,x_\infty,y_\infty\}$ represents the same matroid as the restriction of $M_\mathcal{P}$ to this subset: it is easy to see the images of the five points are distinct, and that the lines spanned by $\{\rho(x_1),\rho(x_\infty)\}$ and $\{\rho(y_1),\rho(y_\infty)\}$ meet at $\rho(O)$.

What remains is to show that the images of circuits of $M_\mathcal{P}$ are dependent.
This is trivial for triples of elements all lying on the flat spanned by $O$ and $x_{\infty}$ or $O$ and $y_{\infty}$ or $x_\infty$ and $y_{\infty}$. In all three cases the resulting subsets are of size at least $3$, and lie on a projective line.

Now consider a triple $\{x_i,y_k,z_j\}$ for some indices $1\le i,k,j\le N$. By construction, such a triple is a circuit of $M_P$ only in the following situations.
\begin{enumerate}
	\item $i=k$ and $j=1$,
	\item $i=j$ and $k=1$,
	\item The equation $X_i = X_j \cdot X_k$ is in $\mathcal{P}$.
\end{enumerate}
Note that in all three cases it is true by definition that $a_i = a_j \cdot a_k$.

Taking representatives in $D^3$ for $\rho(x_i),\rho(y_k),\rho(z_j)$ and writing them as columns, 
we have a dependence
\[\begin{bmatrix}
1 \\ a_i \\ 0
\end{bmatrix} \cdot 1 + 
\begin{bmatrix}
1 \\ 0 \\ a_k
\end{bmatrix} \cdot (-1) +
\begin{bmatrix}
0 \\ a_j \\ -1
\end{bmatrix} \cdot a_k = 
\begin{bmatrix}
0 \\ a_i - a_j \cdot a_k \\ 0
\end{bmatrix} = 0.
\] 

Now suppose $\mathcal{P}$ contains the equation $X_i = X_j + X_k$ for some indices $1\le i,j,k\le N$.
In this case $M_\mathcal{P}$ contains the circuits $\left\{ y_{1},r_{k},x_{\infty}\right\} $, $\left\{ x_{k},r_{k},y_{\infty}\right\} $, and $\left\{ x_{i},r_{k},z_{j}\right\}$.
The images under $\rho$ of each of these circuits are dependent, because
\[\begin{bmatrix}
1 \\ 0 \\ 1
\end{bmatrix} \cdot 1 + 
\begin{bmatrix}
1 \\ a_k \\ 1
\end{bmatrix} \cdot (-1) +
\begin{bmatrix}
0 \\ 1 \\ 0
\end{bmatrix} \cdot a_k = 0,
\] 

\[\begin{bmatrix}
1 \\ a_k \\ 0
\end{bmatrix} \cdot 1 + 
\begin{bmatrix}
1 \\ a_k \\ 1
\end{bmatrix} \cdot (-1) +
\begin{bmatrix}
0 \\ 0 \\ 1
\end{bmatrix} \cdot a_k = 0,
\]
and

\[\begin{bmatrix}
1 \\ a_i \\ 0
\end{bmatrix} \cdot 1 + 
\begin{bmatrix}
1 \\ a_k \\ 1
\end{bmatrix} \cdot (-1) +
\begin{bmatrix}
0 \\ a_j \\ -1
\end{bmatrix} \cdot (-1) = 0,
\]
where in the last equality we used the assumption $a_i = a_j + a_k$.

It is also clear that any subset of size $4$ of the image of $\rho$ is dependent.
\end{proof}

\begin{proof}[Proof of~\Cref{thm:VS}~\eqref{thm:ii}]
Let $\rho$ be a representation of a member of $\mathcal{M_P}$ in a division ring~$D$.
Then by \Cref{lem:proj_equiv}, there exists a projective transformation which takes $O,x_{1},y_{1},x_{\infty},y_{\infty}$
to the elements $[1:0:0]$, $[1:1:0]$, $[1:0:1]$, $[0:1:0]$, and $[0:0:1]$ in $D\mathbb{P}^{2}$, respectively.
By composing with such a transformation, we may assume $\rho$ maps $O,\ldots,y_\infty$ to $[1:0:0],\ldots,[0:0:1]$.

This ensures that for all $1\le i\le N$, $\rho$ maps
\[
	x_i \mapsto [1:a_i:0], \quad y_i \mapsto [1:0:a_i'], \quad z_i \mapsto [0: a_i'':-1],
\]
for some suitable elements $a_i,a_i',a_i''\in D$.

Since $\rho$ represents some matroid in $\M_{\mathcal{P}}$, it maps each circuit of $M_\mathcal{P}$ to a set of rank at most $2$.

Consider the circuit $\{x_i,y_i,z_1\}$ for some $i$, and take representatives in $D^3$ for its image under $\rho$:
\[
\begin{bmatrix}1 \\ a_i \\ 0\end{bmatrix},\quad \begin{bmatrix}1 \\ 0 \\a_i'\end{bmatrix},\quad \begin{bmatrix} 0 \\ 1 \\ -1 \end{bmatrix}.
\]
A dependence between these columns is a solution $\lambda_1,\lambda_2,\lambda_3\in D$ to the equation

\[ \begin{bmatrix}
1 & 1 & 0 \\
a_i & 0 & 1 \\ 
0 & a_i' & -1
\end{bmatrix} \cdot 
\begin{bmatrix}\lambda_1 \\ \lambda_2 \\ \lambda_3 \end{bmatrix} = 0.\]
It is clear that any solution must have $\lambda_1 = -\lambda_2$ and $\lambda_3 = a_i' \cdot \lambda_2$. Thus any nonzero solution has $\lambda_1 \neq 0$, and we may without loss of generality take $\lambda_1 = 1,\lambda_2=-1$, and $\lambda_3=-a_i'$. The columns are known to be dependent, which implies (by considering the middle row of the matrix equation above) that $a_i = a_i'$.

Applying the same considerations to circuits of the form $\{x_i,y_1,z_i\}$ shows $a_i = a_i''$ for each $i$.

Now we claim that substituting $X_0=0$, $X_1=a_1=1$, and $X_i=a_i$ yields a solution to the atomic equations in $\mathcal{P}$.
The proof is essentially the same as the proof that $a_i = a_i'$.

Suppose an equation $X_i=X_j\cdot X_k$ is in $\mathcal{P}$.
By assumption, the matroid represented by $\rho$ contains the circuit $\left\{ x_{i},y_{k},z_{j}\right\} $ in this case. Taking representatives for $\rho(x_i),\rho(y_k),$ and $\rho(z_j)$ in $D^3$ and using the existence of a nontrivial dependence, we have

\[\begin{bmatrix}
1 \\ a_i \\ 0
\end{bmatrix} \cdot \lambda_1 + 
\begin{bmatrix}
1 \\ 0 \\ a_k
\end{bmatrix} \cdot \lambda_2 +
\begin{bmatrix}
0 \\ a_j \\ -1
\end{bmatrix} \cdot \lambda_3 = 0
\] 
where $\lambda_1,\lambda_2,\lambda_3$ are not all $0$. Considering the first row we see $\lambda_1 = -\lambda_2$, and considering the last row we see $\lambda_3 = a_k \cdot \lambda_2$. Thus $\lambda_1 \neq 0$ (or the dependence is trivial) so we may assume $\lambda_1 = 1$. Substituting and considering the second row, we obtain $a_i - a_j\cdot a_k = 0$.

Similarly, consider an equation $X_i=X_j+X_k$ in $\mathcal{P}$.
Then, the matroid represented by $\rho$ contains the element $r_{k}$ together with the circuits $\left\{ y_{1},r_{k},x_{\infty}\right\} $, $\left\{ x_{k},r_{k},y_{\infty}\right\} $, and $\left\{ x_{i},r_{k},z_{j}\right\}$.
The first two circuits imply $\rho(r_{k})= [1:a_k:1]$.
In the same way as for a multiplicative equation in $\mathcal{P}$, we obtain $a_i=a_j+a_k$.

Therefore the tuple $(a_1,\dots,a_N)\in D^N$ solves~$\mathcal{P}$.
\end{proof}

The same proof of \Cref{thm:VS}~\eqref{thm:ii} works mutatis mutandis for \Cref{thm:VS}~\eqref{thm:iii}.
The elements in the division ring are replaced by the invertible matrices in $M_c(\F)$ together with the zero matrix. The proof of \Cref{lem:proj_equiv} also works in this setting (see \cref{rem:2.4_multilinear}).

The solution to $\mathcal{P}$ corresponding to a multilinear representation of a matroid $M\in\mathcal{M_P}$ will always have all matrices invertible
or zero, because if a $c\times c$ matrix $A$ is neither $0$ nor
invertible then 
\[
\mathrm{rk}\begin{bmatrix}I_{c} & I_{c}\\
0 & A\\
0 & 0
\end{bmatrix}
\]
is not a multiple of $c$ (the two columns of this block matrix correspond to the images of $O$ and some element $x_i$ in a representation).

\begin{rmrk}
The converse of ~\Cref{thm:VS}~\eqref{thm:i} fails for multilinear representations due to the failure of certain ranks to be multiples
of $c$: for any $c\ge 2$, there are pairs of invertible $c\times c$ matrices $A,A^{\prime}$
such that
\[
\mathrm{rk}\begin{bmatrix}I_{c} & I_{c}\\
A & A^{\prime}\\
0 & 0
\end{bmatrix}\notin c\mathbb{N}.
\]
It is impossible for a solution of
a system $\mathcal{P}$ having $X_{i}=A$ and $X_{j}=A^{\prime}$
to correspond to a multilinear representation of any $M\in\mathcal{M}_{\mathcal{P}}$ when $A,A'$ have this property.
\end{rmrk}

\section{Undecidability in division rings}\label{sec:undec}
We summarize some results of Macintyre on the undecidability of the word problem in division rings, see \cite{Mac73}, and also \cite{Mac79} for refinements.
Using these, it is easy to deduce the next theorem.

\begin{theorem}\label{thm:undecidability}
	The following problems are undecidable.
	\begin{enumerate}
		\item Given a matroid, decide whether it is skew-linear.
		\item Fix $p$ prime or $0$. Given a matroid, decide whether there exists a division ring of characteristic $p$ over which it is representable.
	\end{enumerate}
	Further, there exists a division ring $D$ such that the following problem is undecidable: given a matroid, decide whether it is representable over $D$.
\end{theorem}

This is similar to the multilinear case -- cf.~\cite{KY19}. For linear matroid representations the situation is simpler: there is an algorithm to decide whether a given matroid is linear, and also an algorithm to decide whether a matroid is linear over a fixed finite or algebraically closed field. The linear characteristic set is also computable, see for instance~\cite[Theorem 6.8.14]{Oxl11}.

\subsection{Macintyre's results}
This section summarizes part of \cite{Mac73} for the reader's convenience.

Let $L$ be the language of group theory, having the constant $1$ and the functions $\cdot, ^{-1}$. Macintyre also uses it for the multiplicative theory of division rings, with the additional convention $0^{-1} = 0$ (this is used only to make sure $^{-1}$ is always defined, and $0$ plays no role).

A universal Horn sentence in $L$ is an expression of the form:
\[\forall x_1 \ldots \forall x_n \left(\bigwedge_{i=1}^m A_i(x_1,\ldots,x_n)=B_i(x_1,\ldots,x_n)\right) \implies A(x_1,\ldots,x_n) = B(x_1,\ldots,x_n)\]
where each expression of the form $A_i(x_1,\ldots,x_n)$ (and similarly $B_i, A, B$) is a product of some sequence of the variables $x_1,\ldots,x_n$ and their inverses. We will call the equations $A_i = B_i$ (excluding $A=B$) the \emph{equations} of the Horn sentence, $A=B$ its \emph{implication}. 

A Horn sentence in $L$ is true in division rings if it is true in every division ring. It is true in groups if it is true in every group.

Given a universal Horn sentence $H$ in $L$, the main theorem of \cite{Mac73} produces a universal Horn sentence $H'$ in $L$, computable from $H$, such that $H$ is true in groups if and only if $H'$ is true in division rings.

The proof contains a construction for each prime number $p$ or $0$, of a division ring $D_p$ of characteristic $p$ with the following property: $D_p$ has elements $\bar{x_1},\ldots,\bar{x_n}$ satisfying the equations of $H'$, and $\bar{x_1},\ldots,\bar{x_n}$ satisfy the implication of $H'$ if and only if $H'$ is true in division rings.
\subsection{\texorpdfstring{Proof of \Cref{thm:undecidability}}{Proof of the undecidability theorem}}
We begin by noting that the truth of a Horn sentence in division rings can be reduced to a sequence of representation problems for matroids.

\begin{lemm}\label{lem:horn}
	Let $\mathcal{D}$ be a class of division rings. If it is decidable whether a matroid is representable over at least one member of $\mathcal{D}$, then the truth of Horn sentences in $\mathcal{D}$ is decidable.
\end{lemm}
This lemma will be applied with $\mathcal{D}$ equal to either a single division ring, the class of all division rings, or the class of division rings having a fixed characteristic.
\begin{proof}
 Given a Horn sentence $H$ in $L$ of the form:
\[\forall x_1 \ldots \forall x_n \left(\bigwedge_{i=1}^m A_i(x_1,\ldots,x_n)=B_i(x_1,\ldots,x_n)\right) \implies A(x_1,\ldots,x_n) = B(x_1,\ldots,x_n),\]
the equations $A_i = B_i$ for $i=1,\ldots,m$ form a system of equations in the variables $x_1,\ldots,x_n$ and their inverses. To this system we add a variable $y$ together with the equation:
\[(A(x_1,\ldots,x_n) - B(x_1,\ldots,x_n))y = 1,\]
to ensure $A = B$ does not hold. 

This yields a system of equations in $L$ which can be solved in at least one division ring $D \in \mathcal{D}$ if and only if $H$ is false in $\mathcal{D}$.

To handle the convention $0^{-1}=0$, we split into $2^n$ cases, one for each subset of the variables. 

Let $S \subset [n]$ (these will be the variables with value $0$). For each $i\in S$, replace each occurrence of $x_i^{-1}$ by $x_i$ and add the equation $x_i = 0$ to the system.
By adding equations of the form
\[x_i x_i' = 1\]
for all $i\notin S$ and replacing each occurrence of $x_i^{-1}$ with $x_i'$ for such $i$, we obtain a polynomial system in $2n-|S|$ noncommuting variables.

Let $\mathcal{M}_S$ be the family of von Staudt matroids associated with this system. This is a finite family which can be computed from $H$. It is clear that some element $M\in \mathcal{M}_S$ is skew-linear over some division ring $D \in \mathcal{D}$ if and only if $H$ is false in $\mathcal{D}$ and there is a counterexample with $\{i\in[n] \mid x_i = 0\}=S$.

Thus, at least one matroid in $\bigcup_{S\subset[n]}\mathcal{M}_S$ is representable over at least one element of $\mathcal{D}$ if and only if $H$ is false in $\mathcal{D}$. Since $\bigcup_{S\subset[n]}\mathcal{M}_S$ is finite and computable from $H$, this proves the lemma.
\end{proof}

The proof of the theorem is now straightforward:

By Macintyre's results, there is a Horn sentence $H'$, computable from $H$, which is true in all division rings if and only if $H$ is true in all groups. This remains true when restricting to division rings of a fixed characteristic. The truth of a Horn sentence in groups is undecidable (this is the uniform word problem). Together with \Cref{lem:horn}, this proves the first two parts of~\Cref{thm:undecidability}.

Let $G=\langle s_1,\ldots,s_n | r_1,\ldots,r_k \rangle$ be a finitely presented group in which the word problem is undecidable. Let $w=w(s_1,\ldots,s_n)$ denote a word in the generators of $G$. Then $w=e$ in $G$ if and only if the Horn sentence $H$ given by
\[\forall s_1 \ldots\forall s_n \left(\bigwedge_{i=1}^k r_i(s_1,\ldots,s_n) = 1\right) \implies w(s_1,\ldots,s_n) = 1\]
is true in groups. By the choice of $G$, there is no algorithm which decides whether $w=e$ as $w$ ranges over all words in the generators. In other words, consider the collection of all Horn sentences with the same equations $r_i(s_1,\ldots,s_n)=1$ as $H$: there is no algorithm to decide whether a sentence from this collection is true in groups.

Using this, we prove the third claim of \Cref{thm:undecidability}: Macintyre's results produce a single division ring $D$, dependent only on the equations $r_i=1$ of $H$ (and not on $w$), in which $H'$ is true if and only if $H$ is true in groups.
Thus, representability of matroids over $D$ is undecidable.

\section{Multilinear characteristic sets}\label{sec:weyl}
\subsection{Introduction to  characteristic sets}
We begin with a comparison of linear, skew-linear, and multilinear characteristic sets of matroids.
\begin{defn}
	The multilinear characteristic set of a matroid $M$ is the set
	\[\{p \mid M \text{ has a multilinear representation over a field of characteristic } p\}.\]
\end{defn}
The linear and skew-linear characteristic sets are defined analogously.

Linear characteristic sets of matroids are well-understood: if $M$ is a matroid, its linear characteristic set is either finite or cofinite.
The latter case occurs if and only if the set contains $0$.
Any finite set of primes is the linear characteristic set of a matroid, and is also complementary to some linear characteristic set. These results are explained in \cite[Section 6.8]{Oxl11}, which refers to Rado, V\'amos, Reid, and Kahn \cites{Rad57}{Vam71}[p.101-2]{BK80}{Kahn82}.

Skew-linear characteristic sets are more complicated: they need not be finite or cofinite, and there is a matroid with skew-linear characteristic set $\{0\}$.
Like in the linear case, an infinite skew-linear characteristic set must contain $0$: this follows from the compactness theorem of first-order logic. See~\cite[Section 3.4]{EH91} for both of these results.

The focus of this article lies on multilinear characteristic sets which are different from both of these.
Like in the linear situation, if a multilinear characteristic set contains $0$ then it is cofinite.
 However, there exists a matroid with multilinear characteristic set equal to the set of all prime numbers, that is, only $0$ is excluded.

We do not yet know whether there exists an infinite multilinear characteristic set with an infinite complement.

\begin{lemm}\label{lem:char_set}
	Let $M$ be a matroid. If the multilinear characteristic set of $M$ contains $0$ then it contains all but finitely many primes.
\end{lemm}
\begin{proof}
	Suppose $M$ has a multilinear representation of order $c$ over a field of characteristic~$0$. Then, just like in the linear case, the multilinear representation matrices of $M$ of the same order $c$ are the solutions of a certain polynomial system with coefficients in $\Z$. The set of characteristics over which such a system can be solved is always finite or cofinite, and contains all but finitely many primes if it contains $0$: this is a theorem in commutative algebra and can be proved for instance using Chevalley's theorem on constructible sets, see~\cite[Appendix B]{BV03} for details.
\end{proof}
The proof can be used to show part of the characterization for linear characteristic sets. It also shows that if $M$ has a multilinear characteristic set which does not contain $0$, then its multilinear representations of any fixed order have only finitely many characteristics.

\subsection{An infinite multilinear characteristic set excluding $0$}
We construct an explicit example of such a characteristic set using the Weyl algebra.
The \emph{Weyl algebra} $W$ over a field~$\F$ is defined to be the free algebra over $\F$ generated by the two elements $X$ and $Y$, modulo the two-sided ideal generated by the element $XY-YX-1$.
It is an Ore domain and admits a division ring of fractions, which we denote by $\widehat{W}$, see\ \cite[Section 6.1]{Coh95} for details.

As detailed in~\Cref{lem:atomic}, there is an atomic system of equations that has a solution in an arbitrary division ring $D$ if and only if the Weyl algebra equation $XY-YX=1$ has a solution in $D$.
In this case, the equations are
\[
X_0=0, \quad X_1= 1, \quad X_2 \cdot X_3 = X_4, \quad X_3 \cdot X_2 = X_5, \quad	X_1 + X_5 = X_4.
\]
We denote the system of these equations by $\mathcal{P}_W$.
Given a solution in a division ring or invertible matrices to the system $\mathcal{P}_W$, setting $X=X_2$ and $Y=X_3$ is a solution to the equation $XY-YX=1$.

\begin{defn}\label{def:matroid}
	We define the \emph{Weyl matroid} $M_W$ to be the principal von Staudt matroid of the atomic system of equations $\mathcal{P}_W$.
\end{defn}
It can be checked that the circuits defining $M_W$ satisfy the circuit elimination axiom and~$M_W$ is therefore actually a matroid in this case.

Part of the Weyl matroid is drawn in~\Cref{fig:weyl_matroid}: the circuits of the equations ${\color{darkgreen}X_2\cdot X_3=X_4}$, $\color{blue}X_3\cdot X_2 = X_5$ and $\color{red}X_1+X_5=X_4$ are depicted as curves and segments.

\begin{theorem}\label{thm:weyl_algebra}
	\begin{enumerate}[(i)]
		\item The Weyl matroid $M_W$ is representable over the division ring $\widehat{W}$.
		\item The Weyl matroid $M_W$ is not multilinear over any field of characteristic $0$.
		Furthermore, the matroid $M_W$ is not multilinear over a field of characteristic $p$ with $p>0$ if the order is not a multiple of $p$.
		\item\label{weyl_thm:iii} The Weyl matroid $M_W$ is multilinear over a field of characteristic $p$ for any $p>0$ and order $p$.
	\end{enumerate}
	In particular, the multilinear characteristic set of $M_W$ is the set of all prime numbers.
\end{theorem}

\begin{figure}
		\centering
		\includegraphics[width=.6\linewidth]{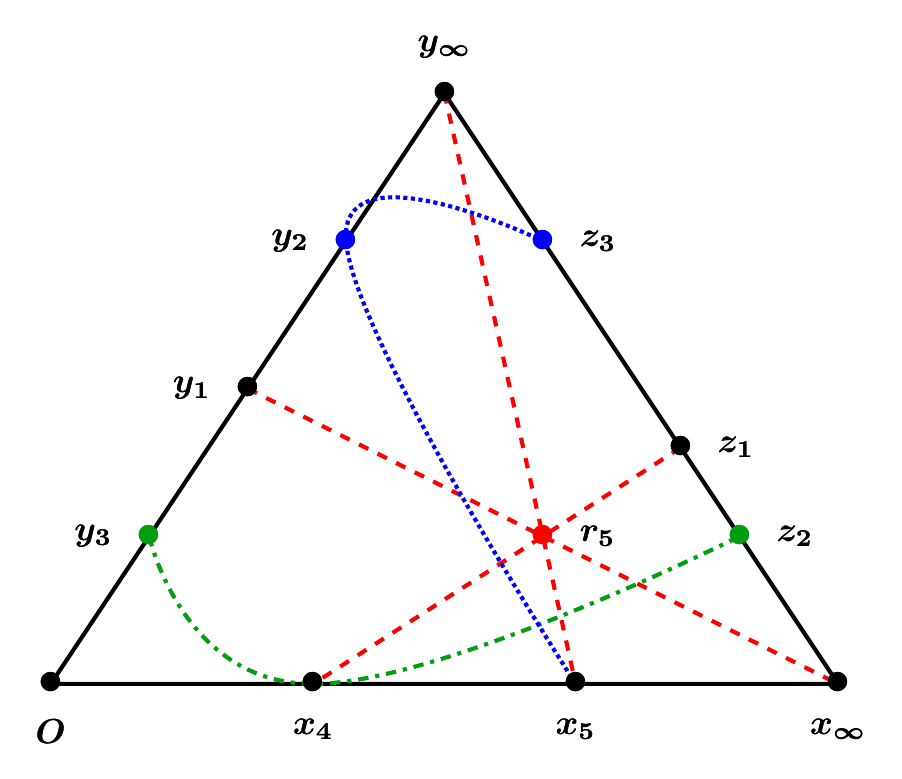} 
	\caption{A geometric representation of a part of the Weyl matroid $M_W$.}
	\label{fig:weyl_matroid} 
\end{figure}

\begin{proof}
	By definition, the equation $XY-YX=1$ has a solution in the Weyl algebra $W$ and therefore also in its division ring of fractions $\widehat{W}$.
	The same holds for its induced atomic system $\mathcal{P}_W$.
	\Cref{thm:VS}~\eqref{thm:i} then implies that some von Staudt matroid of the polynomial system $\mathcal{P}_W$ is representable over the division ring $\widehat{W}$.
	Since $X,Y,XY,YX,1$ are all distinct in~$\widehat{W}$ and do not satisfy any atomic polynomial equation apart from the ones in $\mathcal{P}_W$, this representable matroid over $\widehat{W}$ is in fact the principal von Staudt matroid $M_W$.
	
	Assume $M_W$ is multilinear of order $c$ over some field $\F$.
	\Cref{thm:VS}~\eqref{thm:iii} then implies that there are matrices $A_1,\dots,A_5\in GL_c(\F)\cup \{0\}$ that are a solution to the equations in~$\mathcal{P}_W$.
	Thus, $A_2A_3 - A_3A_2 = I_c$.
	Taking the trace of the matrices in this equation yields
	\[
		c=\tr(I_c)=\tr (A_2A_3-A_3A_2)=\tr(A_2A_3)-\tr(A_3A_2)=0,
	\]
	in $\F$.
	The Weyl matroid $M_W$ is therefore not multilinear of order $c$ for any $c\ge 1$ over any field of characteristic $0$.
	Furthermore, if $\F$ is a field of characteristic $p>0$, the order $c$ must be a multiple of $p$.

	To prove claim~\eqref{weyl_thm:iii}, we note that the division ring of fractions of the Weyl algebra over a field $\F$ of characteristic $p$ is finite-dimensional over its center, and hence embeds in a matrix ring over a field $\mathbb{L}$ extending $\F$. It follows that ranks of matrices over the division ring agree (up to a constant factor $p^2$) with the ranks of corresponding block matrices over $\mathbb{L}$. The rest of the proof is the same as the proof of skew-linearity of the Weyl matroid (using the Weyl algebra over a field of characteristic $p$).
	
	For completeness we also give a detailed computational proof.
	
	We show that $M_W$ is multilinear of order $p$ over the field $\F_p(\lambda,\mu)$, where $p$ is any prime number and $\lambda ,\mu$ are two algebraically independent elements over~$\F_p$.
	Consider the following two $p\times p$ matrices over $\F_p(\lambda,\mu)$:
	\[
	A \coloneqq \begin{bmatrix}
		\lambda & 1 			 & 0 				& \dots & 0 \\
		0 			&	\lambda  & 2 				& \dots & 0 \\
		0			&	 0 			&	\lambda     & \ddots & 0 \\
		\vdots   & \vdots     & \ddots          & \ddots & p-1 \\
		0 			& 0				& 0					& 0 		& \lambda
	\end{bmatrix}, \quad
	B\coloneqq \begin{bmatrix}
	0 			& 0 			 & 0 				& \dots & \mu \\
	1 			&	0			  & 0 				& \dots & 0 \\
	0			&	 1			&	0			     & \dots & 0 \\
	\vdots   & \vdots     & \ddots          & \ddots & \vdots \\
	0 			& 0				& 0					& 1 		& 0
	\end{bmatrix}.
	\]
	A direct computation yields
	\[
	AB\coloneqq \begin{bmatrix}
		1			 & 0 			 & 0 				& \dots & \lambda \mu \\
		\lambda	&	2			  & 0 				& \dots & 0 \\
		0			&	 \lambda &	3     & \ddots & 0 \\
		\vdots   & \vdots     & \ddots          & \ddots & 0 \\
		0 			& 0				& 0					& \lambda		& 0
	\end{bmatrix}, \quad
	BA\coloneqq \begin{bmatrix}
		0			 & 0 			 & 0 				& \dots & \lambda \mu \\
		\lambda	&	1			  & 0 				& \dots & 0 \\
		0			&	 \lambda &	2     & \ddots & 0 \\
		\vdots   & \vdots     & \ddots          & \ddots & 0 \\
		0 			& 0				& 0					& \lambda		& p-1
	\end{bmatrix},
	\]
	which implies $AB-BA=I_p$.
	
	Next consider the following $3\times E_W$ block matrix $W$ of $p\times p$ blocks where $E_W$ is the ground set of $M_W$ and we write $\cdot$ for a $p\times p$ zero matrix for better readability:
	\begin{footnotesize}
	\[
			\begin{bheadmatrix}[
			O & x_{\infty}& y_{\infty}  	 & x_1 	 	& x_2 	 & x_3 		& x_4 	  & x_5   & y_1			& y_2 & y_3 & y_4 & y_5 &  z_1	& z_2  & z_3 & z_4 & z_5 & r_{5}			][*{19}{r}]
			I_p 	& \cdot  & \cdot  	 & I_p     &  I_p   & I_p 	&  I_p    &  I_p	 & I_p 	&  I_p     &  I_p  &  I_p     &  I_p   & \cdot  & \cdot     & \cdot  & \cdot     & \cdot  &   I_p \\
 			\cdot	& I_p  & \cdot & I_p  	 & A 		& B		&  AB  	   & BA  & 	\cdot   & \cdot  & \cdot & \cdot  & \cdot & I_p 	 &  A      &  B   & AB   & BA  & BA	\\
			\cdot & \cdot   & I_p & \cdot   &  \cdot & \cdot  & \cdot  & \cdot& I_p   & A   & B & AB   & BA  &    -I_p    & -I_p   & -I_p  & -I_p   & -I_p & I_p
		  \end{bheadmatrix}.
	\]
\end{footnotesize}

	To show that the matrix $W$ is indeed a multilinear representation of order $c$ of some matroid we need to verify that the rank of all block column minors of $W$ is a multiple of $p$.
	This is trivial to check but slightly tedious.
	Therefore, we omit this part of the proof here.
	In the appendix we prove some parts of this verification as an example.
	The argument relies on the fact that $\lambda$ and $\mu$ are algebraically independent elements over $\F_p$.
	
	As a second step, we claim that $W$ is actually a multilinear representation of the matroid~$M_W$ over $\F_p(\lambda,\mu)$.
	The fact $AB-BA=I_p$ implies that the substitution $X_1=I_p$, $X_2=A$, $X_3=B$, $X_4=AB$, and $X_5=BA$ is a solution to the equations in $\mathcal{P}_W$.
	Further note that the block columns of the matrix $W$ exactly correspond to the images of the map~$\rho$ in the proof of~\Cref{thm:VS}~\eqref{thm:i} after passing from division ring elements to suitable $c\times c$ matrices.
	Thus, the analogous arguments as in this proof show that the multilinear representation given by $W$ respects the circuits of $M_W$ prescribed by the von Staudt constructions.
	That is, the minor of the block columns of a three element circuit of $M_W$ has rank $2p$.
	It can also be verified that the minors of block columns of a triple of $W$ which is not a circuit of $M_W$ has rank $3p$.
	
	Hence, $M_W$ is multilinear over the field $\F_p(\lambda,\mu)$ for all prime numbers $p\ge 2$.
\end{proof}

\section{A skew-linear, nonmultilinear matroid}\label{sec:baumslag}

In this section we prove the following theorem.
\begin{theorem}\label{thm:baumslag}
	There exists a skew-linear matroid which is not multilinear.
\end{theorem}

\begin{defn}\label{def:BS_Group}
	For nonzero integers $m,n$, the \emph{Baumslag--Solitar group} $\mathrm{BS}(m,n)$ is
	\[\langle a,b \mid ba^m b^{-1}  a^{-n}\rangle.\]
\end{defn}

We will work with $\mathrm{BS}(2,3)$, which has the following properties:
\begin{enumerate}\label{prop:BS_group}
	\item\label{prop:a} The group $\mathrm{BS}(2,3)$ is not \emph{residually finite}. That is, there exists an element $w\in \mathrm{BS}(2,3)$ with $w\neq 1$ that is in the kernel of every homomorphism from $\mathrm{BS}(2,3)$ to a finite group. Such an element is given by $w=[bab^{-1},a^{-1}]$, where $[x,y]=xyx^{-1}y^{-1}$ is the commutator of the elements $x$ and $y$.
	\item\label{prop:b} The group $\mathrm{BS}(2,3)$ is a subgroup of the multiplicative group of some division ring~$D_{\mathrm{BS}}$.
\end{enumerate}

Property~\eqref{prop:a} was proved by Meskin in~\cite{Mes72}.
Property~\eqref{prop:b} follows from a theorem of Lewin and Lewin, see~\cite{LL78}. To apply their theorem, we need to know that $\mathrm{BS}(2,3)$ is torsion-free: this is a very special case of Wise's results in~\cite{Wis09}.

The next lemmas recast these properties in a more convenient form.

\begin{lemm}\label{lem:BS_multilinear}
	If $A,B$ are invertible $c\times c$ matrices over a field, and $BA^{2}B^{-1}=A^{3}$, then
	\[BAB^{-1} A^{-1} BA^{-1}B^{-1} A = I.\]
\end{lemm}

\begin{proof}
	Suppose $A,B$ are invertible $c\times c$ matrices over a field satisfying $BA^{2}B^{-1}=A^{3}$. By Mal'cev's theorem (see~\cite{Mal40}), the group $G$ generated by $A,B$ is residually finite (the theorem applies to any finitely-generated group of matrices over a field). Thus, if $g=BAB^{-1} A^{-1} BA^{-1}B^{-1} A \neq I,$ there exists a finite group $H$ and a homomorphism $G\to H$ such that $g$ has a nontrivial image in $H$. Pre-composing such a homomorphism with the homomorphism $\mathrm{BS}(2,3) \to G$ given by $a\mapsto A, b\mapsto B$ gives a contradiction to property~\eqref{prop:a}.
\end{proof}

\begin{lemm}\label{lem:BS_div_ring}
	There exists a division ring $D$ and nonzero elements $a,b\in D$ such that $ba^2 b^{-1} = a^3$ and
	\[bab^{-1} a^{-1} ba^{-1}b^{-1} a \neq 1.\]
\end{lemm}

This is clear from property~\eqref{prop:b} above, together with the fact that $w$ is nontrivial in~$\mathrm{BS}(2,3)$.

\begin{theorem}\label[theorem]{thm:bs_polynomial_system}
	The following polynomial system has a solution in the division ring $D_{BS}$, but has no solution in matrices over any field.
	\begin{equation}\label{eq:baumslag}
		xx' = 1, \quad yy' = 1, \quad yx^2 y' = x^3, \quad z(yxy' x' yx'y'x - 1) = 1.
	\end{equation}
	(In matrices, by $1$ we mean the identity matrix of the appropriate size.)
\end{theorem}
\begin{proof}
	This is a straightforward application of the two previous lemmas: The first two equations mean $x,y$ are invertible with inverses $x',y'$, respectively. The third is then equivalent to $yx^2y^{-1}=x^3$, and the last means $yxy^{-1} x^{-1} yx^{-1}y^{-1}x - 1$ has some multiplicative inverse~$z$, so it is nonzero (and invertible in any division ring).
\end{proof}

Applying \Cref{thm:VS} to the system~\eqref{eq:baumslag} gives \Cref{thm:baumslag}.

More explicitly: One of the von Staudt matroids associated to the system of the previous theorem has a representation in some division ring $D$ (since the system has a solution there).
However, none of these matroids has a multilinear representation, since the system has no solution in matrices.

\appendix
\section{Proof of Theorem \ref{thm:weyl_algebra}, second part}

In this appendix, we complete the proof of~\Cref{thm:weyl_algebra} by checking that all block column minors of the matrix $W$ given in~\Cref{sec:weyl} have rank a multiple of $p$.
We do not show this in all cases but rather present some cases as an example of the general technique.

In the proof, we use the following statements regarding the rank of block matrices.
\begin{lemm}\label{lem:block_matrices}
	Let $\F$ be a field. Let $M_1,M_2,M_3\in M_k(\F)$ be invertible $k\times k$ matrices.
	\begin{enumerate}[(i)]
		\item\label{eq:lem:i} The block matrix $
		\begin{bsmallmatrix}
		I_k & I_k\\
		M_1& M_2\\
		\end{bsmallmatrix}$
		has rank $k+\rk(M_1-M_2)$.
		\item\label{eq:lem:ii} The block matrix $
				\begin{bsmallmatrix}
				I_k      & I_k    &  0\\
				M_1 & 0 &  M_3\\
				0      & M_2     & -I_k 
				\end{bsmallmatrix}$
				has rank $2k+\rk(M_3M_2-M_1)$.	
		\item\label{eq:lem:iii} The block matrix $
		\begin{bsmallmatrix}
		 I_k    &   I_k       &    I_k    \\
		 M_1    & 0     &   M_3  \\
		 0 &  M_2&  I_k
		\end{bsmallmatrix}$
		has rank $2k+\rk(M_1 +M_3M_2 - M_1M_2)$.
	\end{enumerate}
\end{lemm}

The proofs of these statements are by Gaussian elimination via block column operations from the right and are left to the reader.

To prove that all block column minors of the matrix $W$ given in~\Cref{sec:weyl} have rank a multiple of $p$ it suffices to consider pairs and triples of block columns:
Since the rank of any set of block columns is at most $3p$ we can remove some elements of any set of size at least four to a pair or triple with equal total rank.
Thus, it suffices to consider pairs and triples of block columns in the following.
We split up the discussion into two parts.

\begin{prop}
	All pairs and triples of block columns of the matrix $W$ that do not involve $r_{5}$ are of rank $2p$ or $3p$.
\end{prop}
\begin{proof}
	First, we consider pairs of block columns out of the sets $\{x_1,\dots,x_5\}$,  $\{y_1,\dots,y_5\}$, and $\{z_1,\dots,z_5\}$.
	\Cref{lem:block_matrices}~\eqref{eq:lem:i} and slight variations thereof imply that such a pair of block columns is of rank $p+\rk(M_1-M_2)$ where $M_1,M_2$ are two matrices out of the set $\{I_p, A,B,AB,BA\}$.
	 The explicit descriptions of $A,B,AB,BA$ given in~\Cref{sec:weyl} together with the fact that $\lambda$ and $\mu$ are algebraically independent elements over $\F_q$ yields that $M_1-M_2$ is invertible for all $ M_1,M_2 \in \{I_p, A,B,AB,BA\}$ with $M_1\neq M_2$.
	Thus, all pairs of block columns not involving $r_{5}$ are of rank $2p$.
	
	Now consider a triple $\{x_i,y_k,z_j\}$ for $1\le i,j,k \le 5$.
	\Cref{lem:block_matrices}~\eqref{eq:lem:ii} implies that the corresponding block column minor has rank $2p+\rk(M_jM_k-M_i)$ where $M_1,\dots,M_5$ are the matrices $I_p, A, B, AB, BA$, respectively.
	Thus, it can be verified using the explicit description of these matrices in~\Cref{sec:weyl} that the matrix $M_j\cdot M_k-M_i$ is either zero or invertible.
	Hence, the corresponding block column minor has rank $2p$ or $3p$.
\end{proof}

To complete the proof we need to consider pairs and triples involving the element $r_{5}$.

\begin{prop}
	The pairs of block columns of the matrix $W$ that involve $r_{5}$ are of rank~$2p$.
	The block columns of all triples involving $r_{5}$ have rank $2p$ or $3p$.
\end{prop}
\begin{proof}
	Any pair involving $r_{5}$ clearly has rank $2p$ since $r_{5}$ is the only block column with invertible blocks in each row.
	
	So consider any triple $T$ involving $r_{5}$.
	If the triple also contains one of the elements $\{x_0,x_{\infty}, y_{\infty}\}$ the block column minor clearly has rank $3p$.
	Suppose  $T=\{ x_i,y_j,r_5 \}$ for some $1\le i,j\le 5$.
	In this case we can apply~\Cref{lem:block_matrices}~\eqref{eq:lem:iii} to conclude that the block minor $T$ has rank $2p+\rk (M_i+BAM_j-M_iM_j)$ where the matrices $M_1,\dots,M_5$ are again defined to be $I_p, A, B, AB, BA$, respectively.
	Using the explicit description given in~\Cref{sec:weyl} together with the fact that $\lambda$ and $\mu$ are algebraically independent over $\F_p$ one can again verify that the matrix $M_i+BAM_j-M_iM_j$ is invertible for all $1\le i,j\le 5$.
	Hence, the corresponding block column minor has rank $3p$.
	
	The cases  $T=\{ x_i,z_j,r_5 \}$ and $T=\{ y_i,z_j,r_5 \}$ for $1\le i,j\le 5$ can be checked analogously via a variation of \Cref{lem:block_matrices}~\eqref{eq:lem:iii}.
	Therefore, we omit these two last cases.
	The only difference occurs at the triple $\{r_{5},x_4,z_1\}$ which corresponds to a block matrix of rank~$2p$ as required by the circuits in the definition of the von Staudt matroids.
\end{proof}

\bibliographystyle{myalpha}
\bibliography{matroid.bib}

\end{document}